\documentclass[a4paper]{amsart}

\usepackage{graphicx}

\newtheorem{theorem}{Theorem}
\newtheorem{proposition}[theorem]{Proposition}
\newtheorem{lemma}[theorem]{Lemma}

\newtheorem{corollary}[theorem]{Corollary}

\def\R{\mathbb{R}}
\def\C{\mathbb{C}}

\def\Q{\mathbb{Q}}
\def\Z{\mathbb{Z}}

\newcommand{\Oh}{\mathrm{O}}
\newcommand{\oh}{\mathrm{o}}
\newcommand{\e}{\mathrm{e}}
\newcommand{\im}{\mathrm{i}}
\newcommand{\dd}{\mathrm{d}}

\begin{document}

\title{Asymptotic Estimates for Some Number Theoretic Power Series}

\author{Stefan Gerhold}
\thanks{This work was financially supported by CDG, BA-CA, AFFA, and the joint centre Microsoft Research-INRIA}

\address{Vienna University of Technology and
Microsoft Research-INRIA, Orsay}

\email{sgerhold at fam.tuwien.ac.at}




\date{\today}

\begin{abstract}
  We derive asymptotic bounds for the ordinary generating functions of several classical
  arithmetic functions, including the M\"obius, Liouville, and von Mangoldt functions.
  The estimates result from the Korobov-Vinogradov zero-free region for the Riemann zeta-function,
  and are sharper than those obtained by Abelian theorems from bounds for the summatory functions.
\end{abstract}

\maketitle

\vskip 5mm
\begin{small}
\begin{flushright}
{\sl Such functions as $\sum \mu(n) x^n$, \\
$\sum \phi(n) x^n$, $\sum \Lambda(n)x^n$ are extremely difficult to handle.}
\\[2mm]
{\sc --- G.H. Hardy, E.M. Wright} \rm\cite{HaWr79}
\end{flushright}
\end{small}


\section{Introduction}

While Hardy and Wright are of course right in that ordinary generating functions of arithmetic
functions do not share the versatility and usefulness of their well-known Dirichlet counterparts,
several non-trivial results -- both old and new -- have been obtained for them. For instance, the analysis
of
\[
  \sum_{n=1}^\infty \tau(n) z^n = \sum_{n=1}^\infty \frac{z^n}{1-z^n},
\]
where~$\tau(n)$ denotes the number of divisors of~$n$, goes back to Lambert~\cite{Kn13},
and the expansion
\begin{align}
  & \sum_{n=1}^\infty \tau(n) \e^{-n t} \sim \frac1t \log \frac1t + \frac{\gamma}{t}
    - \sum_{n=0}^\infty \frac{B_{n+1}^2}{(n+1)!(n+1)} t^n, \label{eq:tau exp} \\
  & \text{where $t \to 0$},\quad |\arg(t)|<\tfrac12 \pi-\theta\quad \text{for some $\theta>0$}, \notag
\end{align}
involving Euler's constant and Bernoulli numbers, has been known for a long time~\cite{FlGoDu95,Sc74,Titchmarsh86,Wi17}.
Titchmarsh~\cite{Titchmarsh86} has applied~\eqref{eq:tau exp} in a result on mean values
of the Riemann zeta-function, and
Canfield et~al.~\cite{CaSaWi04} have extended~\eqref{eq:tau exp} to the case of the arithmetic
function that counts only divisors in some fixed residue class.
Another generalization has been obtained by Berndt and Evans~\cite{BeEv85}, who also proved the formula
\[
  \sum_{n=1}^\infty p_n z^n \sim \frac{1}{(1-z)^2} \log \frac{1}{1-z}, \qquad z\to 1^-\ \text{in}\ \R,
\]
where~$p_n$ is the $n$-th prime number.

Recently, the transcendence of number theoretic power series has been of interest to several authors.
Banks et~al.~\cite{BaLuSh05} have established the irrationality of $\sum \mu(n)z^n$, $\sum p_n z^n$, $\sum \tau(n)z^n$,
and several other similar series, over~$\Q(z)$.
Later it was noted~\cite{BeCo09,BoCo09} that the transcendence of these series follows easily from the fact that they
have the unit circle as a natural boundary.
This property even shows that they are not $D$-finite~\cite{BeGeKlLu08,FlGeSa05,St80}.
General results about the transcendence of $\sum f(n) z^n$ with~$f$ multiplicative have recently been
obtained by Borwein and Coons~\cite{BoCo09} and by Bell and Coons~\cite{BeCo09}.

The present note is concerned with asymptotic estimates for power series $\sum a_n z^n$,
where the Dirichlet generating function $\sum a_n n^{-s}$ has singularities at the zeros of the Riemann zeta-function.
For instance, $a_n=\mu(n)$ falls under this category.
Delange~\cite{De00} has noted that the prime number theorem in the form
\[
  M(x) := \sum_{n\leq x} \mu(n) = \oh(x), \qquad x\to\infty,
\]
where~$M(x)$ denotes the Mertens function, readily implies
\[
  \sum_{n=1}^\infty \mu(n) z^n = \oh\left( \frac{1}{1-z} \right), \qquad z\to1^-\ \text{in}\ \R.
\]
A quick way to improve this starts from Walfisz' deep result~\cite{Wa63}
\begin{equation}\label{eq:walfisz}
  M(x) = \Oh\left(x \exp\left( -\frac{c(\log x)^{3/5}}{(\log\log x)^{1/5}} \right) \right).
\end{equation}
Recall the following a basic Abelian theorem~\cite{BiGoTe89,FlGeSa05,Se95}:
\begin{lemma}\label{le:abelian}
  Suppose that~$(a_n)$ is an ultimately monotone real sequence with $a_n\sim n^\alpha \ell(n)$,
  where $\alpha>0$, and~$\ell$ is positive and varies slowly at infinity. Then
  \[
    \sum_{n=1}^\infty a_n z^n \sim \frac{\Gamma(\alpha+1)}{(1-z)^{\alpha+1}} \ell\left( \frac{1}{1-z}\right)
  \]
  as $z\to1$ in any sector
  \begin{equation}\label{eq:sector}
    S_\theta := \{ z \in\C : |\arg(1-z)| \leq \tfrac12 \pi -\theta \}, \qquad \theta >0.
  \end{equation} 
\end{lemma}
From the lemma (applied here only for real~$z$) and~\eqref{eq:walfisz} we obtain
\begin{align}
  \sum_{n=1}^\infty \mu(n) z^n &= (1-z) \sum_{n=1}^\infty M(n) z^n \notag \\
  &= \Oh\left(\frac1t \exp\left( -\frac{c(\log 1/t)^{3/5}}{(\log\log 1/t)^{1/5}} \right) \right),
  \qquad t=-\log z \sim 1-z \to0^+\ \text{in}\ \R. \label{eq:from abel}
\end{align}
There seems to be no Tauberian result available to translate~\eqref{eq:from abel} back into an estimate
for the Mertens function~$M(x)$. This typical asymmetry suggests that we might be able to do a little better
than~\eqref{eq:from abel} by using dedicated methods. Indeed, our main result
(Theorem~\ref{thm:main} below) improves~\eqref{eq:from abel} to
\begin{equation}\label{eq:mu est}
  \sum_{n=1}^\infty \mu(n) z^n = \Oh\left(\frac1t \exp\left(-\frac{0.0203\times \log (1/t)}
    {(\log\log 1/t)^{2/3}(\log\log\log 1/t)^{1/3}}\right) \right),
   \quad t=-\log z,
\end{equation}
where~$z\to0$ in an arbitrary sector~$S_\theta$, $\theta>0$.
The proof rests on the contour integral representation~\cite{Titchmarsh86} 
\begin{equation}\label{eq:mu repr}
  \sum_{n=1}^\infty \mu(n) \e^{-nt} =
    \frac{1}{2\pi\im} \int_{\kappa-\im\infty}^{\kappa+\im\infty} \frac{\Gamma(s)}{\zeta(s)} t^{-s} \dd s, \qquad \kappa>1.
\end{equation}
In a way that is familiar from the prime number theorem or the Selberg-Delange method~\cite{Tenenbaum95},
one can deform the integration contour a little bit into the critical strip~$0<\Re(s)<1$,
and then estimate the resulting integral.
The exponential decrease of the Gamma function along vertical lines is a convenient feature
of~\eqref{eq:mu repr}, which is not present in the Perron summation formula~\cite{Tenenbaum95}
\begin{equation}\label{eq:perron}
  \sum_{n\leq x}\mu(n) =  \frac{1}{2\pi\im} \int_{\kappa-\im\infty}^{\kappa+\im\infty} \frac{1}{\zeta(s)}\frac{x^s}{s}  \dd s,
     \qquad \kappa>1,\ x \in \R^+\setminus\Z.
\end{equation}
Power series thus tend to be easier to estimate than summatory functions. The fact that~$x$
is real in~\eqref{eq:perron}, whereas in~\eqref{eq:mu repr} it is natural to consider also complex~$t$,
causes no great difficulties. (At least if~$|\arg(t)|$ stays bounded away from~$\tfrac12\pi$.)

In the following section we put~\eqref{eq:mu est} into perspective by relating the growth
of $\sum \mu(n) z^n$ to the Riemann Hypothesis.
Section~\ref{se:main} contains our main result, from which~\eqref{eq:mu est} follows.
A few related power series will be estimated in Section~\ref{se:fu ex}.
Section~\ref{se:open} collects some open problems.


\section{Connection to the Riemann Hypothesis}

In conjunction with~\eqref{eq:walfisz} and~\eqref{eq:mu est}, the following proposition shows that the gap
between the Riemann Hypothesis and what is provable today is slightly smaller in the power series
case than in the case of the summatory function~$M(x)$.

\begin{proposition}
  Let $\tfrac12 \leq\eta <1$. Then the following are equivalent:
  \begin{itemize}
    \item[$(i)$]\label{it:i} $\zeta(s)$ has no zeros for $\Re(s)>\eta$,
    \item[$(ii)$] $M(x)=\Oh(x^{\eta+\varepsilon})$,
    \item[$(iii)$] $\sum_{n\geq1} \mu(n) z^n = \Oh((1-z)^{-\eta+\varepsilon})$ as $z\to 1^-$ in $\R$.
  \end{itemize}
\end{proposition}
\begin{proof}
  The equivalence of~$(i)$ and~$(ii)$ is classical for~$\eta=\tfrac12$, see Titchmarsh~\cite{Titchmarsh86},
  and the proof of the more general case is an easy modification.
  (The implication $(ii)\Rightarrow (i)$, which we actually do not require, is posed as Exercise~13.4
  in Apostol's textbook~\cite{Apostol76}.)
  If~$(ii)$ holds, then~$(iii)$ follows by Lemma~\ref{le:abelian}.
  Finally, if we assume that~$(iii)$ is true, we have that
  \[
    F(t) := \sum_{n=1}^\infty \mu(n) \e^{-nt} =\Oh(t^{-(\eta+\varepsilon)}), \qquad t\to0^+\ \text{in}\ \R.
  \]  
  Hence the Mellin transform~\cite{FlGoDu95}
  \[
    \int_0^\infty F(t) t^{s-1} \dd t = \frac{\Gamma(s)}{\zeta(s)}
  \]
  defines an analytic function for~$\Re(s)>\eta$.
\end{proof}

Under the Riemann Hypothesis, one would expect that we can push the integration contour in~\eqref{eq:mu repr}
across the critical line~$\Re(s)=\tfrac12$ to obtain  an expansion of the form
\begin{equation}\label{eq:mu rh}
  \sum_{n=1}^\infty \mu(n) \e^{-nt} \stackrel{?}{=} t^{-1/2} H(\log(1/t)) -2 + \oh(1),
  \qquad t\to0,
\end{equation}
where~$H$, a bounded oscillating function, is a sum of infinitely many harmonics corresponding to
the non-trivial zeros of the zeta-function. The fast decrease of the Gamma function makes the residues
of $\Gamma(s)/\zeta(s)$ at these zeros rather small, so that the term~$-2$ will dominate in~\eqref{eq:mu rh}
unless~$1-z$ is very close to zero. Indeed, the $\Omega(t^{-1/2})$ term becomes numerically visible
only from about $1-z=10^{-10}$ onwards [P.~Flajolet, private communication].
This ``fake asymptotics'' property has also been noted by Bateman and Diamond~\cite{BaDi00}.
Without assuming the Riemann Hypothesis, Delange~\cite{De00} has shown that
\[
  \sum_{n=1}^\infty \mu(n) z^n = \Omega_\pm\left( \frac{1}{\sqrt{1-z}} \right), \qquad z\to 1^-\ \text{in} \ \R,
\]
which is in line with~\eqref{eq:mu rh}, and shows that the left-hand side {\em does not} converge to~$-2$.


\section{Main Result}\label{se:main}

We write
\[
  D(s) = \sum_{n=1}^\infty\frac{a_n}{n^s}, \qquad s = \sigma + \im \tau,
\]
for the Dirichlet generating function of a sequence~$a_n$. The following theorem gives an estimate
for the power series $\sum a_n z^n$ near $z=1$, assuming analyticity and growth conditions for~$D(s)$.
\begin{theorem}\label{thm:main}
  Let~$a_n$ be a sequence of complex numbers such that~$D(s)$ is absolutely convergent for~$\Re(s)>1$
  and has an analytic continuation to a set~$\Omega$ of the form
  \begin{equation}\label{eq:domain}
    \sigma \geq g(\tau) :=
    \begin{cases}
      1- b(\log|\tau|)^{-\alpha} (\log \log |\tau|)^{-\beta} & |\tau| \geq w \\
      1- b(\log w)^{-\alpha} (\log \log w)^{-\beta} &  |\tau| \leq w
    \end{cases}
  \end{equation}
  for some positive parameters~$\alpha,\beta,b,w$. Assume furthermore that
  \[
    D(s) = \Oh(\tau^\nu),
  \]
  uniformly as $s\to\infty$ in~$\Omega$, for some~$\nu>0$. Then for any $\varepsilon>0$
  \begin{equation*}\label{eq:gen est}
    \sum_{n=1}^\infty a_n z^n = \Oh\left(\frac1t \exp\left(-\frac{(b-\varepsilon) \log (1/t)}
     {(\log\log 1/t)^{\alpha}(\log\log\log 1/t)^{\beta}}\right) \right),
     \qquad t=-\log z \sim 1-z.
  \end{equation*}
  The variable~$z$ may tend to~$1$ in an arbitrary sector of the form~\eqref{eq:sector}.
\end{theorem}
This result immediately implies the bound~\eqref{eq:mu est},
by noting that $D(s)=1/\zeta(s)$ for $a_n=\mu(n)$ and putting $\alpha=\tfrac23$ and $\beta=\tfrac13$.
 The required analyticity and growth
of~$1/\zeta(s)$ are the content of Korobov and Vinogradov's famous theorem~\cite{Titchmarsh86}, which describes
the largest known zero-free region for the Riemann zeta function.
(Recall that it leads to the best known error term in the prime number theorem.)
For the constant~$b$ in~\eqref{eq:domain} one may take~$b=0.05507\times(4.45)^{-2/3}> 0.0203$ in this case, by
a result of Ford~\cite{Fo02}.
\begin{proof}[Proof of Theorem~\ref{thm:main}]
  The convergence assumption on~$D(s)$ clearly implies that the radius of convergence of
  $\sum a_n z^n$ is at least one.
  We assume that~$z$ stays inside~$S_{2\theta}$; then $t=-\log z$ satisfies
  \[
    |\arg (t)| \leq \tfrac12 \pi - \theta
  \]
  for small~$|t|$.
  For~$\kappa>1$ we have~\cite[p.~151]{Titchmarsh86}
  \begin{equation}\label{eq:int}
     \sum_{n=1}^\infty a_n \e^{-nt} =
      \frac{1}{2\pi\im} \int_{\kappa-\im\infty}^{\kappa+\im\infty} D(s)\Gamma(s)t^{-s} \dd s.
  \end{equation}
  \begin{figure}[t]
    \includegraphics[scale=1.0]{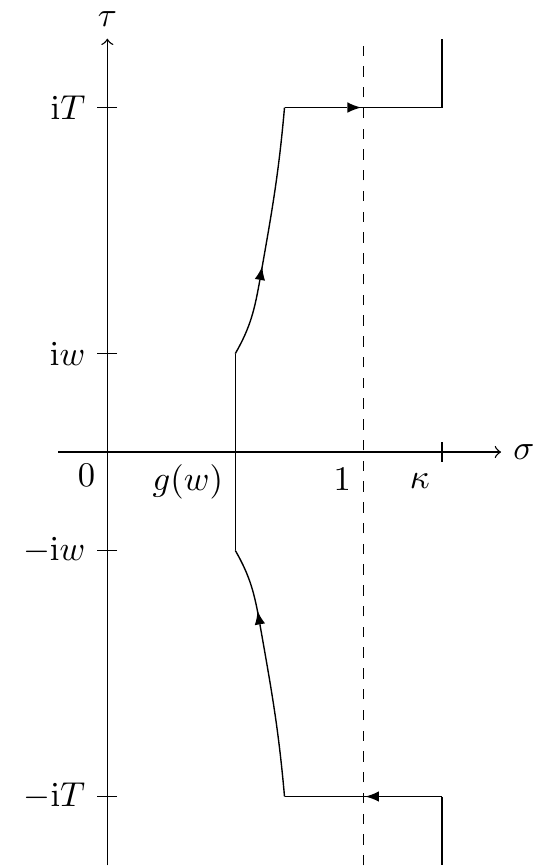}
    \caption{The deformed integration contour, extending into the strip $0<\Re(s)<1$.}
    \label{fig:contour}
  \end{figure}
  Now we deform the integration contour as indicated in Figure~\ref{fig:contour}, where
  \[
    \kappa = 1 - 1/\log|t| >1,
  \]
  and~$T=T(t)>0$, to be be fixed later, tends to infinity as~$t\to0$.
  Between~$\pm\im T$, the contour is defined by the function~$g(\tau)$ from~\eqref{eq:domain}.
 
  We will repeatedly apply the following version of Stirling's formula~\cite{Co35}: If~$\Re(s)=\sigma$
  is confined to a finite interval, then
  \[
    |\Gamma(s)| \sim \sqrt{2\pi}\ \e^{-\pi |\tau|/2}|\tau|^{\sigma-1/2}, \qquad |\tau|\to\infty,
  \]
  uniformly w.r.t.~$\sigma$.

  To bound the integral over the upper vertical line $[\kappa+\im T,\kappa+\im\infty[$, note that there we have
  \begin{align*}
    |\Gamma(s) t^{-s}| &\ll_t  |t|^{-\kappa} \tau^{\kappa-1/2}\e^{-\tau(\pi/2-\arg(t))} \\
    &\leq  \e |t|^{-1} \tau^{\kappa-1/2}\e^{-\theta\tau}.
  \end{align*}
  (Here and in the following, we write~$A \ll_t B$ for $A=\Oh(B)$ as~$t\to0$, where
  the estimate holds uniformly in~$\tau$, if~$\tau=\Im(s)$ is a free variable in the right-hand side~$B$.)
  Hence the integral over the upper vertical line satisfies
  \begin{equation}\label{eq:vert line}
    |I_{\rm{vert}}| \ll_t \frac{1}{|t|} \int_T^\infty \e^{-\theta \tau} \tau^{\kappa+\nu-1/2} \dd \tau
   \ll_t |t|^{-1} \e^{-\theta T} T^{\kappa+\nu-1/2}.
  \end{equation}
  We next estimate the contribution of the horizontal segment $[g(T)+\im T, \kappa +\im T[$
  to the integral. In this range we have
  \[
    |t^{-s}| \leq |t|^{-\kappa} \e^{T(\pi/2-\theta)} = |t|^{-1} \e^{T(\pi/2-\theta)+1}
  \]
  and
  \[
    |\Gamma(s)| \ll_t  T^{\kappa-1/2}\e^{-T\pi /2},
  \]
  hence this portion of the integral is
  \begin{equation}\label{eq:hor}
    |I_{\rm{hor}}| \ll_t |t|^{-1} \e^{-\theta T} T^{\kappa+\nu-1/2},
  \end{equation}
  so that we obtain the same estimate as in~\eqref{eq:vert line}.

  Finally, we bound the integral over the arc $\sigma=g(\tau)$, which we call~$I_{\rm{arc}}$.
  The integral from~$g(w)$
  to~$g(w)+\im w$ is plainly~$\Oh(t^{\delta-1})$ for some positive~$\delta$, hence negligible
  compared to~\eqref{eq:hor}. In the remaining range~$\tau>w$, we have
  \[
    |t^{-s}| = |t|^{-g(\tau)} \e^{\tau \arg(t)} \leq |t|^{-g(T)} \e^{\tau (\pi/2-\theta)}
  \]
  and
  \[
    |\Gamma(s)| \ll_t  \tau^{g(\tau)-1/2} \e^{-\tau\pi/2} \leq  \tau^{g(T)-1/2} \e^{-\tau\pi/2},
  \]
  so that we have the bound
  \begin{align}
    |I_{\rm{arc}}| &\ll_t |t|^{-g(T)} T^\nu
      \int_w^T \e^{-\tau \theta} \tau^{g(T)-1/2} \dd \tau \notag \\
    &\ll_t  |t|^{-g(T)}  T^\nu \Gamma(g(T)+\tfrac12) \notag \\
    &\ll_t  |t|^{-g(T)} T^\nu. \label{eq:arc}
  \end{align}
  To complete the proof, we have to pick~$T$ wisely in order to balance the estimates~\eqref{eq:hor}
  and~\eqref{eq:arc}. We would like to have~$T$ as large as possible in~\eqref{eq:hor},
  whereas~\eqref{eq:arc} calls for a small~$T$.
  We therefore choose
  \[
    T = \frac {\log (1/|t|)}{(\log \log 1/|t|)^\alpha},
  \]
  which makes~\eqref{eq:hor} and~\eqref{eq:arc} approximately equal. The former then implies
  \[
    |I_{\rm{hor}}| \ll_t \frac1t \exp\left(- \frac{(\theta-\varepsilon)\log (1/t)}{(\log \log 1/t)^\alpha} \right),
  \]
  whereas~\eqref{eq:arc} yields
  \[
    |I_{\rm{arc}}| \ll_t \frac1t \exp\left( -\frac{(b-\varepsilon)\log (1/t)}
      {(\log \log 1/t)^\alpha (\log \log \log 1/t)^\beta} \right),
  \]
  both for arbitrarily small $\varepsilon>0$.
\end{proof}


\section{Further Examples}\label{se:fu ex}

Besides~\eqref{eq:mu est}, Theorem~\ref{thm:main} yields also estimates for other number theoretic
power series. In what follows, we let $\Lambda,\lambda,\omega,$ and~$\tau$ denote, as usual, the von Mangoldt function,
the Liouville function, the number-of-distinct-prime-factors function, and the number-of-divisors function.
Applying Theorem~\ref{thm:main} to the Dirichlet generating functions~\cite{De00,Tenenbaum95}
\begin{align}
  \sum_{n=1}^\infty \frac{(-1)^{n+1}\mu(n)}{n^s} &= \frac{1}{\zeta(s)} \frac{2^s+1}{2^s-1}, \label{eq:dgf1} \\
  \sum_{n=1}^\infty \frac{\Lambda(n)-1}{n^s} &= -\frac{\zeta'(s)}{\zeta(s)} - \zeta(s), \\
  \sum_{n=1}^\infty \frac{\lambda(n)}{n^s} &= \frac{\zeta(2s)}{\zeta(s)}, \\
  \sum_{n=1}^\infty \frac{(-1)^{n+1}\lambda(n)}{n^s} &= (1+2^{1-s})\frac{\zeta(2s)}{\zeta(s)}, \\
  \sum_{n=1}^\infty \frac{2^{\omega(n)}-\tau(n)}{n^s} &= 
    \frac{\zeta(s)^2}{\zeta(2s)} - \zeta(s)^2 \label{eq:dgf2}
\end{align}
yields the following result.

\begin{corollary}\label{cor:other fu}
  Let~$E(z)$ denote the function in the error term in~\eqref{eq:mu est}. Then we have
  \begin{align}
    \sum_{n=1}^\infty (-1)^n \mu(n) z^n &= \Oh(E(z)), \notag \\
    \sum_{n=1}^\infty \Lambda(n) z^n &= \frac{1}{1-z} + \Oh(E(z)), \notag \\
    \sum_{n=1}^\infty \lambda(n) z^n &= \Oh(E(z)), \notag \\
    \sum_{n=1}^\infty (-1)^n \lambda(n) z^n &= \Oh(E(z)), \notag \\
    \sum_{n=1}^\infty 2^{\omega(n)} z^n &= \frac{1}{1-z} \log \frac{1}{1-z} +
      \frac{1+\gamma}{1-z} + \Oh(E(z)), \label{eq:2^omega}
  \end{align}
  as~$z$ tends to~$1$ in an arbitrary sector of the form~\eqref{eq:sector}.
\end{corollary}
\begin{proof}
  The Dirichlet series~\eqref{eq:dgf1}--\eqref{eq:dgf2} satisfy the assumptions of Theorem~\ref{thm:main};
  see, e.g., Titchmarsh~\cite{Titchmarsh86}. 
  As for the case of~$2^{\omega(n)}$, formula~\eqref{eq:tau exp} provides the required
  expansion of $\sum \tau(n)z^n$.
\end{proof}
Recall that Selberg and Delange~\cite[II.5]{Tenenbaum95} established expansions for
summatory functions $\sum_{n\leq x}a_n$ in the scale
\begin{equation}\label{eq:log scale}
   x (\log x)^{\rho-k}, \qquad k = 1,2,\dots,
\end{equation}
assuming that the corresponding Dirichlet series $\sum a_n n^{-s}$ is sufficiently
close to a power~$\zeta(s)^{-\rho}$ of the zeta-function, where~$\rho\in\C$.
This is proved from Perron's summation formula, using a contour akin to
Figure~\ref{fig:contour}, but circumventing the possible singularity at~$s=1$ by a narrow loop.
The same programme could be carried out for power series, too, but this seems not worthwhile.
Note that Dirichlet series with a pole at $s=1$ can be handled by Theorem~\ref{thm:main}
after subtracting a singular element, as we did in the proof of~\eqref{eq:2^omega}.
An algebraic singularity at $s=1$ leads to an infinite expansion in the scale~\eqref{eq:log scale},
which readily translates into an expansion for~$\sum a_n z^n$ at $z=1$ by
an Abelian theorem (Lemma~\ref{le:abelian}).


\section{Open Problems}\label{se:open}

As noted in the introduction, the unit circle is a natural boundary of~$\sum\mu(n)z^n$. Hence one would
expect that, if~$z$ tends to~$1$ along a path that comes very close to the unit circle,
the function picks up too much growth from neighboring singularities
to be bounded in any scale involving only $1/(1-z)$.
So the restriction of~$z$ to sectors in Theorem~\ref{thm:main} is presumably essential.
More precisely, we pose the following
question: If $f:\R^+\to\R^+$ is an arbitrary function, does it follow that
\[
  \sum_{n=1}^\infty \mu(n) z^n = \Omega\left( f\left(\frac{1}{1-z}\right) \right)
\]
as~$z\to1$ in the unit disk?

On another register, a natural continuation of the transcendence results mentioned in the
introduction would be to investigate whether the power series $\sum f(n) z^n$, with~$f$
any of the classical arithmetic functions, can satisfy an algebraic differential equation~\cite{Ru89}.

\bigskip

{\bf Acknowledgement.} I thank Philippe Flajolet and Florian Luca for helpful comments.

\bibliographystyle{siam}
\bibliography{../gerhold,../algo}

\end{document}